\documentclass[graybox]{svmult}

\usepackage{type1cm}        
\usepackage{makeidx}         
\usepackage{graphicx}        
\usepackage{multicol}        
\usepackage[bottom]{footmisc}

\usepackage{amsmath}

\usepackage{newtxtext}       %
\usepackage[varvw]{newtxmath}       

\usepackage{algorithm}
\usepackage{algpseudocode}
\usepackage{multicol}

\usepackage{float}
\floatstyle{ruled}
\newfloat{algorithm}{h}{loa}
\floatname{algorithm}{Algorithm}
\algnewcommand\AND{\textbf{ and }}
\usepackage[font=small,skip=3pt]{caption}
\usepackage{hyperref}

\usepackage{tikz}
\usetikzlibrary{patterns,positioning}
\usepackage{pgfplots}
\pgfplotsset{compat=newest}
\usepgfplotslibrary{fillbetween,groupplots}
\usepackage[dvipsnames]{xcolor}

\makeindex             
\usepackage{bm}
\def\dom{\Omega}
\def\domS{\Omega_\mathrm{S}}
\def\domM{\Omega_\mathrm{M}}
\def\bou{\partial\Omega}
\def\bouE{B_\mathrm{ext}}
\def\bouI{B_\mathrm{int}}
\def\Oc{\mathcal{O}}

\def\e{{\bm{\epsilon}}}
\def\g{{\bm{g}}}

\def\w{{\bm{w}}}
\def\x{{\bm{x}}}

\def\z{{\bm{z}}}

\def\Mb{{\bm{M}}}
\def\Mbh{{\bm{M}^{1/2}}}
\def\Mbih{{\bm{M}^{-1/2}}}

\def\N{\mathbb{N}}
\def\R{\mathbb{R}}

\def\Cc{\mathcal{C}}
\def\Fc{\mathcal{F}}
\def\Nc{\mathcal{N}}
\def\Jc{\mathcal{J}}
\def\Oc{\mathcal{O}}
\def\Sc{\mathcal{S}}
\def\Uc{\mathcal{U}}

\def\Om{\Omega}

\def\Finf{\Fc}
\def\Ffin{F}

\newcommand{\inner}[3]{{\left\langle #1, #2 \right\rangle_{#3}}}
\newcommand{\dual}[3]{{\left( #1, #2 \right)_{#3^*,#3}}}

\def\Cp{\Cc_{0}}
\def\Cph{\Cc_{0}^{1/2}}
\def\pst{{\mathrm{post}}}
\def\Cpst{\Cc_{\pst}}

\def\Gmn{\Gamma}

\def\Gmni{\Gamma^{-1}}
\def\Gmnih{\Gamma^{-1/2}}

\def\ResS{R_0}

\DeclareMathOperator{\Diag}{Diag}

\DeclareMathOperator{\tr}{tr}

\renewcommand{\d}{\,\mathrm{d}}

\newcommand{\compt}{\hookrightarrow\mathrel{\mspace{-15mu}}\rightarrow}

\newcommand{\blue}[1]{\textcolor{blue}{#1}}

\newcommand{\abstr}{Within the field of optimal experimental design, \emph{sensor placement} refers to the act of finding the optimal locations of data collecting sensors, with the aim to optimise reconstruction of an unknown parameter from finite data. In this work, we investigate sensor placement for the inverse problem of reconstructing a heat source given final time measurements. Employing forward and adjoint analysis of this PDE-driven model, we show how one can leverage the first author's recently invented \emph{redundant-dominant $p$-continuation} algorithm to obtain binary A-optimal sensor placements also for this time-dependent model.}

\begin{document}

\title*{FEM-based A-optimal sensor placement for heat source inversion from final time measurement\\
}
\titlerunning{FEM-based A-optimal sensor design for heat}
\author{Christian Aarset\orcidID{0000-0001-8163-9305}\\ and Tram Thi Ngoc Nguyen\orcidID{0000-0002-7245-7611}}
\institute{To be submitted to the Proceedings of Domain Decomposition Methods in Science and Engineering XXIX. \\[1ex]
Christian Aarset \at University of G\"ottingen - Institute for Numerical and Applied Mathematics, Lotzestr. 16-18, D-37083 G\"ottingen, Germany, \email{c.aarset@math.uni-goettingen.de}
\and Tram Thi Ngoc Nguyen \at Max Planck Institute for Solar Systems Research, Justus-von-Liebig-Weg 3, 37077 G\"ottingen, Germany \email{nguyen@mps.mpg.de}}
%
%
\maketitle

\abstract*{\abstr}

\abstract{\abstr}



\section{Experimental design for heat source identification}\label{sec:setup}

This article concerns itself with optimal experimental design (OED) for the best reconstruction in the inverse heat source problem. The heat equation allows us to model the spread of heat over time from a heat source and through a spatial domain, by ways of solving a time-dependent partial differential equation (PDE). Inverting the heat equation then serves to reconstruct and identify the unknown heat source, which is of great interest in controlling and optimizing heating and cooling processes for various applications, e.g.~in materials science.

A natural question arises: what is the optimal configuration of sensors, or \emph{optimal design} for data collection, yielding the best possible reconstruction of the unknown heat source? Indeed, it is generally only possible to obtain local approximations of the dispersing heat at a finite number of sample locations; thus, the choice of measurement points is crucial to the quality of the source reconstruction. 

We assume that in the spatial domain, a total of $m$ different fixed candidate locations are proposed for sensor placement, while due to budget constraints, one can place no more than $m_0<m$ of all potential sensors. One strategy would be to repeat the experiment $\binom{m}{m_0}$ times and choose the sensor configuration that yielded the best reconstruction in practice. However, this quickly becomes unfeasible for even modest values of $m$ and $m_0$, due to combinatorial scaling; moreover, it is desirable to fix the design a priori, avoiding the need for a test-trial phase for data collection.

\begin{figure}
    \centering
    \includegraphics[width=0.45\linewidth]{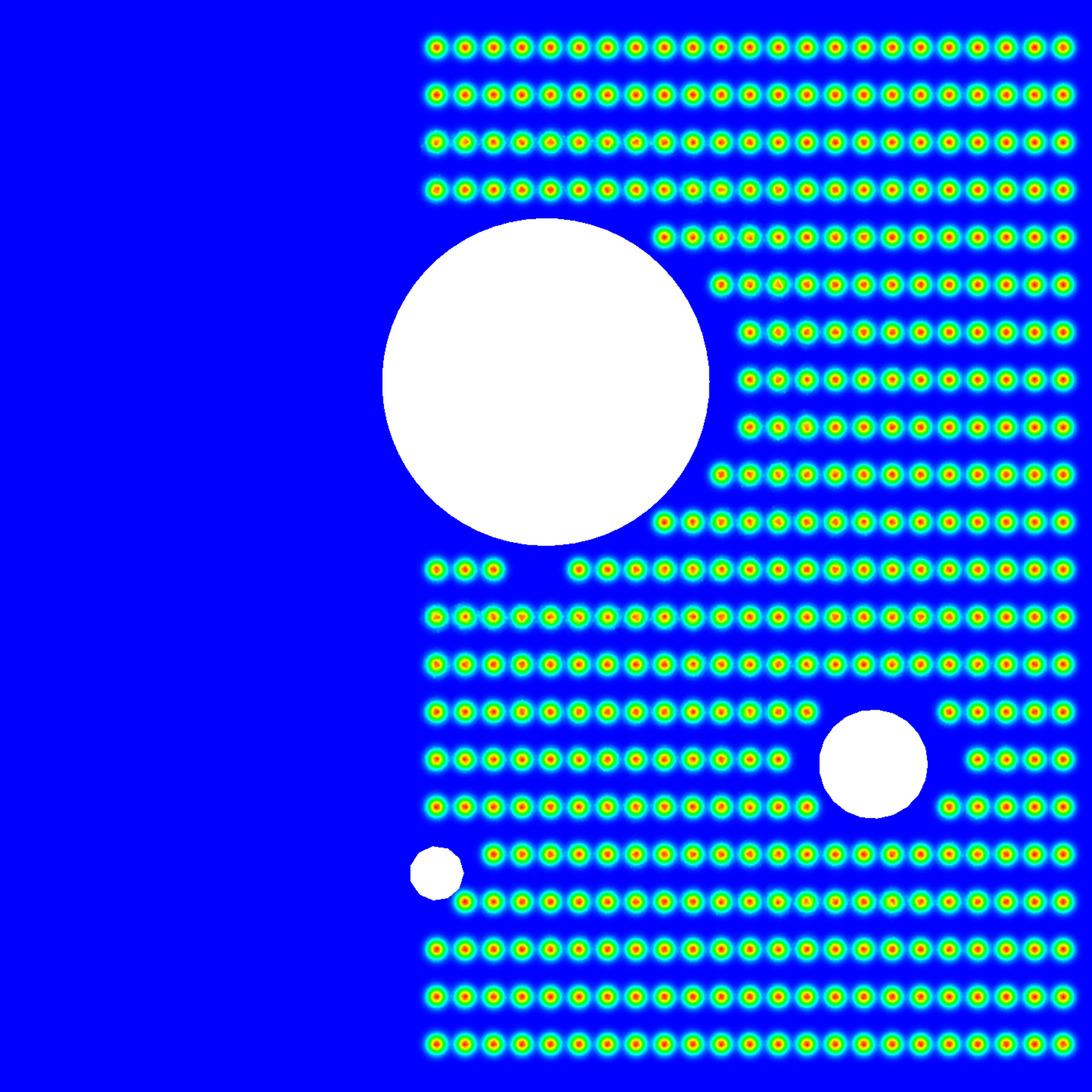}
    \includegraphics[width=0.45\linewidth]{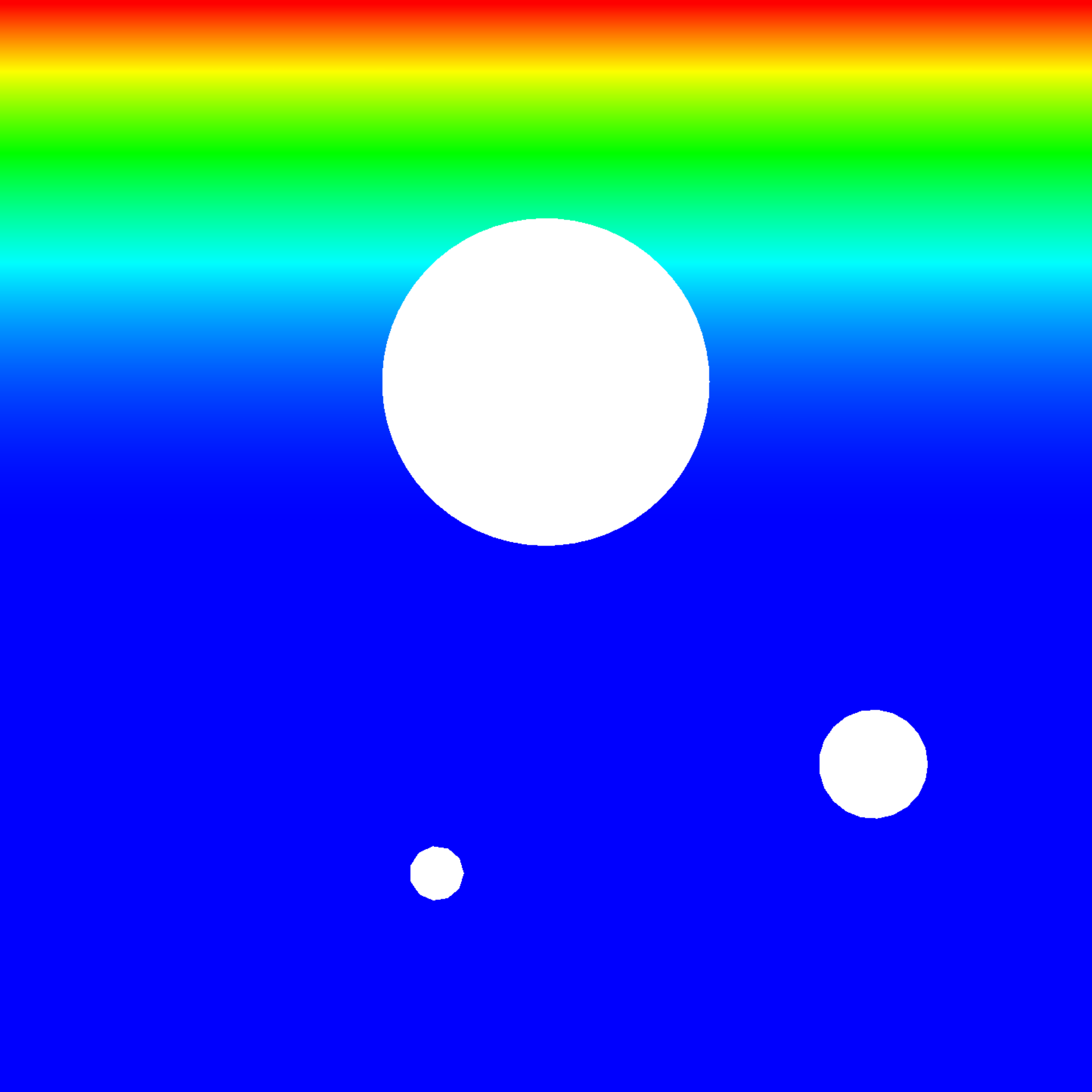}
    \caption{Left: Source domain , full domain and grid of all candidate sensor locations. Right: Spatially dependent part of the diffusion parameter $a$ in \eqref{eq}.} 
    \label{fig:grid}
\end{figure}

As we consider a final time measurement scenario, we assume that each sensor provides only a single scalar observation of the (local approximation to the) heat at its placement location at a fixed measurement time $T>0$.

In what follows, we will present an adaptation of the low-rank approximation-based \emph{redundant-dominant $p$-continuation algorithm} recently invented by the first author in \cite{Aarset} to obtain A-optimal designs, that is, sensor placements that minimise the trace of the posterior covariance, or, equivalently, that minimise average posterior pointwise uncertainty in the reconstruction, see Section \ref{sec:A-optim}. For a deeper discussion of this and related topics, we refer to \cite{Alexanderian}. 

In so doing, we demonstrate the applicability of the results in \cite{Aarset} to time-dependent PDEs, showcasing its relevance to a broad class of problems. To this end, several key elements must be taken into account: well-posedness of the forward heat model, efficiency of the backpropagator and the adaptation of the finite element method (FEM) framework to the time-dependent PDE. These aspects, and their introduction to the A-optimal computational framework, form the contribution of this article.

\textbf{Notation.} 
Finite-dimensional vectors are consistently denoted by boldface lowercase letters (i.e.~$\g$, $\w$), both scalars and function space variables are denoted by plainface lowercase letters (e.g.~$p$, $s$), matrices are denoted by plainface uppercase letters (e.g.~$\Ffin$, $\Gmn$), while operators with infinite-dimensional inputs or outputs are denoted by calligraphic uppercase letters (e.g.~$\Sc$, $\Oc$).

\section{A-optimal sensor placement} \label{sec:A-optim}
\textbf{Design-dependent forward map.} The \emph{sensor placement problem} can be viewed as the problem of finding the experimental design $\w\in\{0,1\}^m$, $m\in\N$ allowing for the best reconstruction of a source $s$ in some Hilbert space $X$, given the design-dependent forward map
\begin{align}\label{Fw}
\Finf_\w: X\to \R^m, \qquad \Finf_\w: = \Diag(\w)\Finf, \qquad \Finf:X\to\R^m,
\end{align}
where $\Diag(\w):\R^m\to\R^m$ is the multiplication operator, represented as a diagonal matrix. As $\w \in \{0, 1\}^m$, the design $\w$ acts as a mask on the data, representing the act of sensor placement. The quantity $m$ thus represents the number of candidate locations where the experimenter might elect to place a sensor.

\textbf{A-optimality.}
For Bayesian linear inverse problems, an \emph{A-optimal design} can be expressed as a sensor placement optimising the reconstruction by minimising the trace of the posterior covariance. We recall from \cite{AttiaConstantinescu, Stuart} that if the unknown $s$ is endowed with the Gaussian prior $\Nc(s_0,\Cc_0)$, $s_0\in X$, $\Cc_0:X^*\to X$, and if the measurement noise is corrupted by additive Gaussian noise $\e\sim\Nc(0,\Gmn)$ in $\R^m$, where $\Gmn\in\R^{m\times m}$ is diagonal and positive definite, then for each design $\w$, the posterior distribution of $s$ given observed noisy data $\g\in\R^m$ is $s|\g\sim\mathcal{N}(s_\pst(\w),\Cpst(\w))$, with
\begin{align}
    s_\pst(\w) & := s_0 + \Cpst(w)\Finf_w^*\Gmni\left(\g-\Finf_\w s_0\right), \\
    \Cpst(\w) & := \left(
        \Finf_w^*\Gmnih\Diag(\w)\Gmnih\Finf_w+\Cc_0^{-1}
    \right)^{-1}.
\end{align}

The A-optimal objective can then be expressed as the trace of the posterior covariance, being a function of the design $\w$, that is,
\begin{align}\label{a-criterion}
    \Jc: \w\in\R^m
    \mapsto \tr\left(
        \Cpst(\w)
    \right) =
    \tr\left(\left(
        \Finf_w^*\Gmnih\Diag(\w)\Gmnih\Finf_w+\Cc_0^{-1}
    \right)^{-1}\right);
\end{align}
which is well-defined for all non-negative and sufficiently small negative values of $\w$. An A-optimal design $\w^*$ using no more than $m_0$ sensors then satisfies 
\[
    \w^*\in\mathrm{argmin}_{\w\in\{0,1\}^m, \|\w\|_0\leq m_0}\Jc(\w).
\]


\section{Heat model and the forward map}\label{sec:heat}
\textbf{Linear heat equation. }
Consider the room $\dom:=(-1,1)^2$ with boundary $\bou=\bouI\cup\bouE$, where both the surfaces $\bouI$ of the rods permeating the domain and the exterior boundaries $\bouE$ are thermally insulated, which we model via homogeneous Neumann boundary conditions. 

The heat distribution $u\in\Uc$ in the room from a spatial-variant heater $s$ is governed by the linear parabolic equation with varying boundary conditions
\begin{alignat}{2}\label{eq}
\begin{cases}
&\dot{u} - \nabla\cdot(a\nabla u)= s \quad\text{in $\dom\times(0,T)$},\\
&\partial_n u  = 0 \quad\text{on $\bou\times[0,T]$,}\\
&u(t=0)  = 0 \quad \text{on \,$\dom$,} 
\end{cases}
\end{alignat}
where the spatially dependent source $s\in L^2(\domS)=:X$ is implicitly extended by zero outside $\domS$. 
Here, $\dot{u}$ denotes the time derivative of $u$, the diffusion parameter is  $a(\x)=1+(5\x_2^5+\x_2^3)\chi_{\x_2\geq 0}\in C(\overline{\dom})$, see Figure \ref{fig:grid}, and $n$ indicates the outward normal vectors at the boundaries.

The PDE \eqref{eq} is well-posed in the sense that for any source $s\in X$ (and initial condition in $L^2(\dom)$), there exists a unique weak solution to \eqref{eq}  with
\begin{align}\label{wellposed}
 u\in\Uc:= L^2(0,T;H^1(\dom))\cap H^1(0,T;H^1(\dom)^*),\quad \|u\|_\Uc\leq C\|s\|_X 
\end{align}
for some  constant $C$ (c.f.~\cite{Evans, TCC}). We thus define the linear source-to-final-state map
\begin{align}\label{S}
    \Sc:X\to L^2(\dom), \quad s\mapsto u(\cdot,T)
\end{align} 
which is well defined due to compactness of the embedding $\Uc\compt C(0,T;L^2(\dom))$ \cite[Lemma 7.3]{Roubicek}.

\textbf{Finite measurement. } A central conceit of sensor placement is the idea that the state can only be measured by representing it as a -- possibly rather small -- vector of observables. As elaborated upon in \cite{Aarset}, this is ubiquitously carried out by pairing with appropriate dual elements. Explicitly, this requires composition with the \emph{finite measurement operator} $\Oc: L^2(\dom)\to\R^m$, given as
\[
    \Oc u := \left(
        \inner{o_k}{u}{L^2(\dom)}
    \right)_{k=1}^m \in \R^m
\]
for all $u\in L^2(\dom)$. Here, the functions $(o_k)_{k=1}^m\subset L^2(\dom)$ can be thought of as \emph{sensors}, and typically represent sufficiently smooth approximations of Dirac delta functions centered away from the source domain. This allows $\Oc$ to act via a series of approximate pointwise measurements of the final time state away from the source domain. In practice, this can be achieved by taking a grid of measurement points $(\x_k)_{k=1}^m\subset\dom\backslash\domS$ with sufficient distance from $\domS$, and then letting each $o_k\in L^2(\dom)$ be only implicitly defined by the action of a finite element software's approximation of pointwise evaluation at $\x_k$.

This formulation finally leads us to the \emph{source-to-observable} map $\Finf:X\to\R^m$,
\[
    \Finf s := \left[\Oc\circ\Sc\right]s = \left(
        \inner{o_k}{\Sc s}{L^2(\dom)}
    \right)_{k=1}^m \in \R^m \qquad \text{for all $s\in X$.}
\]

\section{FEM-based A-optimality via the RedDom-$p$  algorithm}

The \emph{redundant-dominant $p$-continuation algorithm} introduced by the first author in \cite{Aarset} stands as a robust, high-performing algorithm for sensor placement. 
It moreover leverages global optimality theory \cite[Thm.~3]{Aarset} to mark certain sensors as \emph{dominant}, i.e.~always on, and \emph{redundant}, i.e.~always off, enabling significant dimensionality reduction of the OED problem.

\begin{algorithm}[h!]
\caption{Binary OED by $p$-continuation via redundant-dominant classification}\label{alg:p_cont}
\begin{algorithmic}[1]
\Require Initialisation as globally optimal non-binary design $\w^{0}:=\w^*$, power $p=1$, iteration $i=0$, continuation parameter $\delta\in(0,1)$
\While {$\w^i$ has entries significantly different from $0$ and $1$}
	\State $p \gets (1-\delta)p$ and $i\gets i+1$
	\State Solve the non-convex constrained optimization problem (e.g.~via the SLSQP algorithm)
	\begin{align} \label{eq:p_cont:objective}
	\Jc^p(\z) & := \Jc(\z^{1/p}), \\
	\nabla \Jc^p(\z) & = \frac{1}{p}\nabla\Jc(\z^{1/p})\z^{1/p-1}, \label{eq:p_cont:grad}\\
	0 \leq \z_k \leq 1 \qquad & \text{for all $k\in\N$, $k\leq m$}, \qquad
	\sum_{k=1}^m \z_k \leq m_0, \label{eq:p_cont:constr}
	\end{align}
	initialised at $\z := (\w^{i-1})^p$ and keeping dominant and redundant indices ($\w^*_k=1$, resp.~$\w^*_k=0$) fixed, returning $\z^i$
	\State $\w^i \gets (\z^i)^{1/p}$
\EndWhile
\Ensure Binary design $\overline{w}^{m_0}:=\w^i$ as approximate binary optimal design.
\end{algorithmic}
\end{algorithm}

This algorithm requires the adjoint of the  source-to-observable map $\Finf$, which is derived in the following. 

\begin{lemma}[Backpropagator] 
The Banach space adjoint $\Finf^*:\R^m \to X^*$ is given by
\begin{align}\label{eq-adj}
& \Finf^*g=\int_0^T \ResS z\, dt \quad\text{where }
    \begin{cases}
-\dot{z} - \nabla\cdot(a\nabla z)= 0 \quad\text{in $\dom\times(0,T)$},\\
 \partial_n z  = 0 \text{ on $\bou\times[0,T]$,}\\
z(t=T)  = \sum_{k=1}^m g_ko_k \quad \text{on \,$\dom$},
\end{cases}
\end{align}
with $\ResS:L^2(\dom)\to L^2(\domS)$ being the restriction operator.
\end{lemma}

\begin{proof}The restriction $\ResS$ is clearly well-defined. With $s\in X$, $\g\in\R^m$, write $u(T):=\Sc s\in L^2(\dom)$, implying $\dot{u}-\nabla\cdot(a\nabla u)=s$ in $\domS\times(0,T)$ and $0$ elsewhere. We write the inner product via partial integrations in Bochner spaces, that is
\begin{align*}
&\inner{\Finf s}{\g}{\R^m} = 
\sum_{k=1}^m
\inner{o_k}{\Sc s}{L^2(\dom)}\g_k = \sum_{k=1}^m\int_\dom u(T)\g_ko_k\d x\\
&=\int_{\dom}u(T)\left(
\sum_{k=1}^m \g_ko_k
\right)\,dx+\int_0^T\int_{\dom}u( -\dot{z} - \nabla\cdot(a\nabla z))\,dx\,dt\\
&=\int_{\dom}u(T)\left(
\left(
    \sum_{k=1}^m \g_ko_k
\right)
-z(T)\right)\,dx + \int_{\dom} u(0)z(0)\,dx+\int_0^T\int_{\dom} ( \dot{u} - \nabla\cdot(a\nabla u)) z\,dx\,dt \\
&= 0 + 0 + \dual{\int_0^T z\,dt}{s}{X} 
\end{align*}
where $z$ solves the adjoint equation \eqref{eq-adj}, as $s$ is independent of time and is supported on $\domS$.
Above, we make use of homogeneous boundary conditions of $u,z$ and the Bochner integral identity \cite[Lemma 7.3]{Roubicek}: $(u(T),z(T))_{L^2}-(u(0),z(0))_{L^2}=\int_0^T \langle\dot{u}(t),z(t)\rangle - \langle u(t),\dot{z}(t)\rangle\, dt$ for $u,z\in\Uc$. 

\end{proof}


While forward and adjoint evaluations $\Finf s$, $\Finf^*g$ both require solving time-dependent PDEs, they can both, especially in the context of OED, be viewed as mappings between finite element coefficients and vector observables (or vice versa). Explicitly, 
we replace them by their respective halves of the prior-preconditioned misfit Hessian, writing $F\in\R^{m\times n}$,
$$
\Ffin:=\Gmnih \Finf\Cph\Mbh, \qquad \Ffin^T=\Mbih\Cph\Finf^*\Gmnih.
$$ 
The half-powers of the mass matrix $\Mb$ of the source finite element space act as isometric isomorphisms between the finite element space and standard Euclidean space, see \cite{Aarset}. We obtained low-rank decompositions $\Ffin^T\approx QR$ and $\Ffin\approx R^TQ^T$, with $Q\in\R^{n\times\ell}$, $R\in\R^{\ell\times m}$, $\ell=50\ll\min\{m,n\}$ by repeated forward and adjoint evaluations with the randomised subspace iteration algorithm \cite{HalkoMartinssonTropp}. This choice of $\ell$ was made automatically, as it obtained a ratio of $10^{-12}$ between the largest and smallest approximate singular values of $\Ffin$. We remark that as the time-integration in the computation of the adjoint $\Finf^*$ via \eqref{eq-adj} must be done via numerical integration, and the randomised subspace iteration algorithm is quite sensitive to the numerical accuracy of the adjoint, the value of $\Delta t$ must be chosen low to ensure good approximation; however, once the low-rank decomposition has been computed, there is no additional cost associated with the low time-step size, as the low-rank form maps directly from source coefficients to final-time measurements without intermediate calculations at other time-steps.

\begin{lemma}[Algorithm \ref{alg:p_cont} -- Low-rank]\label{low-rank}
The low-rank approximation of the objective \eqref{eq:p_cont:objective} and gradient \eqref{eq:p_cont:grad} in Algorithm \ref{alg:p_cont} for A-optimal design \eqref{a-criterion} take the form
\begin{align}
    \Jc(\w) & = \tr\left(\Cp\right) - \tr\left(C\right) +
    \tr\left(
        \left(R\Diag(\w)R^T+I_\ell\right)^{-1}C
    \right), \\
    \nabla\Jc(\w) & = -\mathrm{ColumnNorm}\left(
        C^{1/2}\left(R\Diag(\w)R^T+I_\ell\right)^{-1}R
    \right),
\end{align}
where $C:=Q^T\Mbh\Cp\Mbih Q\in\R^{\ell\times\ell}$ and $\mathrm{ColumnNorm}$ maps a (finite-dimensional) matrix to the vector containing the norms of each of its columns. 
\end{lemma}
\begin{proof}
We refer to \cite[Theorem 8]{Aarset}.  
\end{proof}

\begin{remark}
With the low-rank algorithm in Lemma \ref{low-rank}, the redundant-dominant $p$-continuation algorithm \ref{alg:p_cont} can thus be carried out efficiently, and unlike e.g.~the greedy algorithm, can be computed fully in parallel.    
\end{remark}

\section{Numerical results}

To obtain our numerical results, we take a FEM discretisation of $X=L^2(\dom)$ as a second-order finite element space with $n_\mathrm{co}=17727$ degrees of freedom, containing $X$ as a subset with $n=5289$ degrees of freedom. Simultaneously, we obtain FEM discretisations of the forward and adjoint maps $\Finf$, $\Finf^*$ via the \textsc{NGSolve} software's \textsc{ngsxfem} integration and the associated time-space finite element spaces \cite{Schoberl}, taking time-steps of $\Delta t=10^{-4}$ over the time interval $T=[0,1]$.

As prior for our Bayesian inversion, we choose $f\sim\mathcal{N}(0,\Cc_0)$, where $\Cc_0:=(-\alpha\Delta + I)^{-2}$ is the densely defined bounded linear inverse bilaplacian operator on $L^2(\domS)$, for both inversions endowed with the Robin boundary condition $\frac{\partial}{\partial n}u=\frac{\sqrt{\alpha}}{1.42}u$, $\alpha=1/4$; this prior is chosen as a moderately smoothing, trace class prior, whose boundary condition is known to lead to spatially uniform prior variance; see \cite{DaonStadler}.

Noise on the finite measurement was realised as $1\%$ of the average variance of one thousand data samples $(\Finf s^i)_{i=1}^n$, with each $s^i$ drawn from the prior distribution, echoing the strategy employed in \cite{Aarset}.

Our low-rank Algorithm \ref{alg:p_cont} within Lemma \ref{low-rank} deliver the sequence $(\overline{\w}^{m_0})_{m_0=1}^{36}\subset\{0,1\}^m$ of binary approximate A-optimal designs. Figure \ref{fig:designs} visualises the optimal designs versus the \emph{posterior pointwise variance field} $c:\Om\to\R$ $c(\x):=\left[\Cpst(w)\delta_{\x}\right](\x)$; this is an excellent visualisation of the uncertainty reduction provided by the design due to the relationship $\tr(\Cpst(\w))=\int_\Om c(x)\d x$, see \cite{Mercer}.

\begin{figure}
    \centering
    \includegraphics[width=0.45\linewidth]{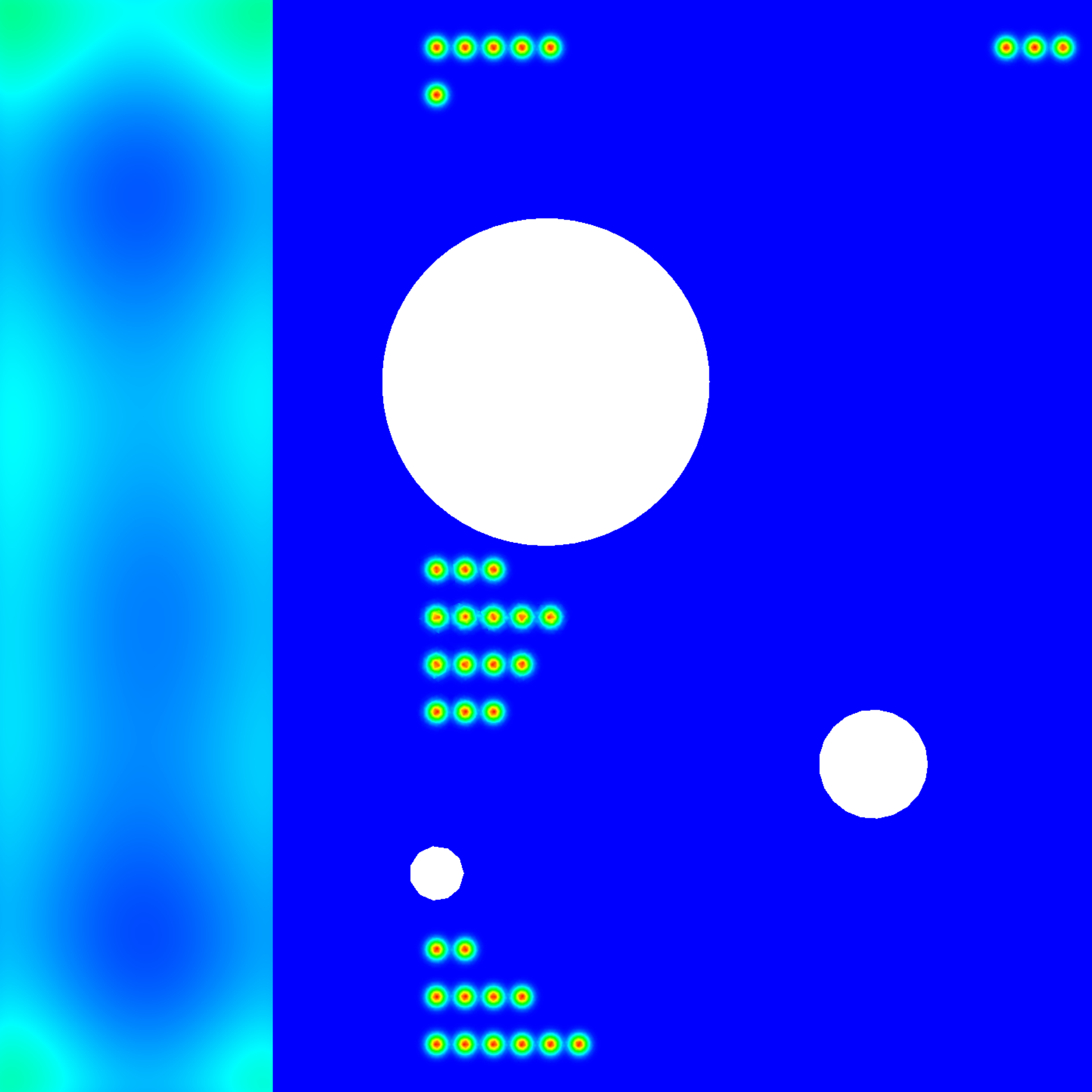}
    \caption{Optimal experimental  for $m_0=36$, plotted together with its corresponding pointwise variance field.}
    \label{fig:designs}
\end{figure}

As a point of comparison, Figure \ref{fig:objective_comparison} displays the A-optimality of the sequence $(\overline{\w}^{m_0})_{m_0=1}^{36}$ compared to that of random designs: For each $m_0$, $10^3$ random designs were drawn, and their A-optimalities are plotted in red. 

\begin{figure}[H]
\centering
\begin{tikzpicture}[scale=0.8]
\begin{axis}[
	no markers, 
	xmin = 1, 
	xmax = 36,
	enlargelimits = 0.01, 
	xlabel = {target number of sensors $m_0$}, 
	xlabel style = {font = \large}, 
	ylabel = {A-optimality $\Jc$}, 
	ylabel style = {font = \large}, 
	ylabel shift=-0pt, 
	xlabel shift=-0pt,
 	xshift=-20pt,
	legend style={at={(0.42,0.8)},anchor=south west}
]

\addplot[color=red!40, name path=maxes, forget plot] table[x=targets,y=randommax, col sep=comma]{graphics/Aoptimalities.csv};

\addplot[color=red!40, name path=mins, forget plot] table[x=targets,y=randommin, col sep=comma]{graphics/Aoptimalities.csv};

\addplot[color=OrangeRed!40]fill between[of=maxes and mins];

\addlegendentry{Random designs}

\addplot[very thick, color=ForestGreen, name path=ws] table[x=targets,y=w, col sep=comma]{graphics/Aoptimalities.csv};

\addlegendentry{Algorithm 1 output}

\addplot[very thick, color=blue, name path=w0s] table[x=targets,y=w1, col sep=comma]{graphics/Aoptimalities.csv};

\addlegendentry{Lower bound}

\end{axis}
\end{tikzpicture}
\caption{Blue: Theoretical lower bound for A-optimality via \cite[Thm.~3]{Aarset}. Green: Proposed binary A-optimal designs. Red: $10^3$ random designs.}
\label{fig:objective_comparison}
\end{figure}
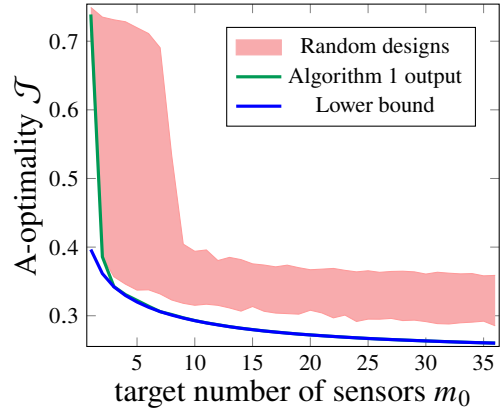

To study one of the designs more in depth, we fix $m_0=36$. 
Figure \ref{fig:MAPs} shows the maximum a posteriori estimates of $s$ from noisy  measurement data $\g:=\Finf_{\w} s + \e$, $\e\sim\Nc(0,\Gmn)$ with the artificial source $s$ given as
$$
\begin{cases}
    \exp(-1/r_{\pm}(\x)^{1/8}), & r_{\pm}(\x) > 0, \\
    0, & \text{else,}
\end{cases} \quad r_{\pm}(\x):=\left(\frac{3\pi}{40}\right)^2-(\x_1+0.75)^2-(\x_2 \pm 0.7)^2;
$$
on the left, $\w:=\overline{\w}^{m_0}$, that is, our proposed approximate optimal design, while on the right, $\w:=\w_{\mathrm{RNG}}^{m_0}$ is set equal to the best random design for $m_0=36$. The relative errors in the reconstructions, are, respectively, $\|s-s_{\overline{\w}^{m_0}}\|_2/\|s\|\approx 0.6491$ and $\|s-s_{\w_{\mathrm{RNG}}^{m_0}}\|_2/\|s\|\approx 0.6782$ for the two reconstructions; while both errors are large, which is not unexpected with such limited measurement data, the reconstructions are qualitatively good, and our proposed solution attains $4.3\%$ lower relative error than the random design.

\begin{figure}
    \centering
    \includegraphics[width=0.32\linewidth]{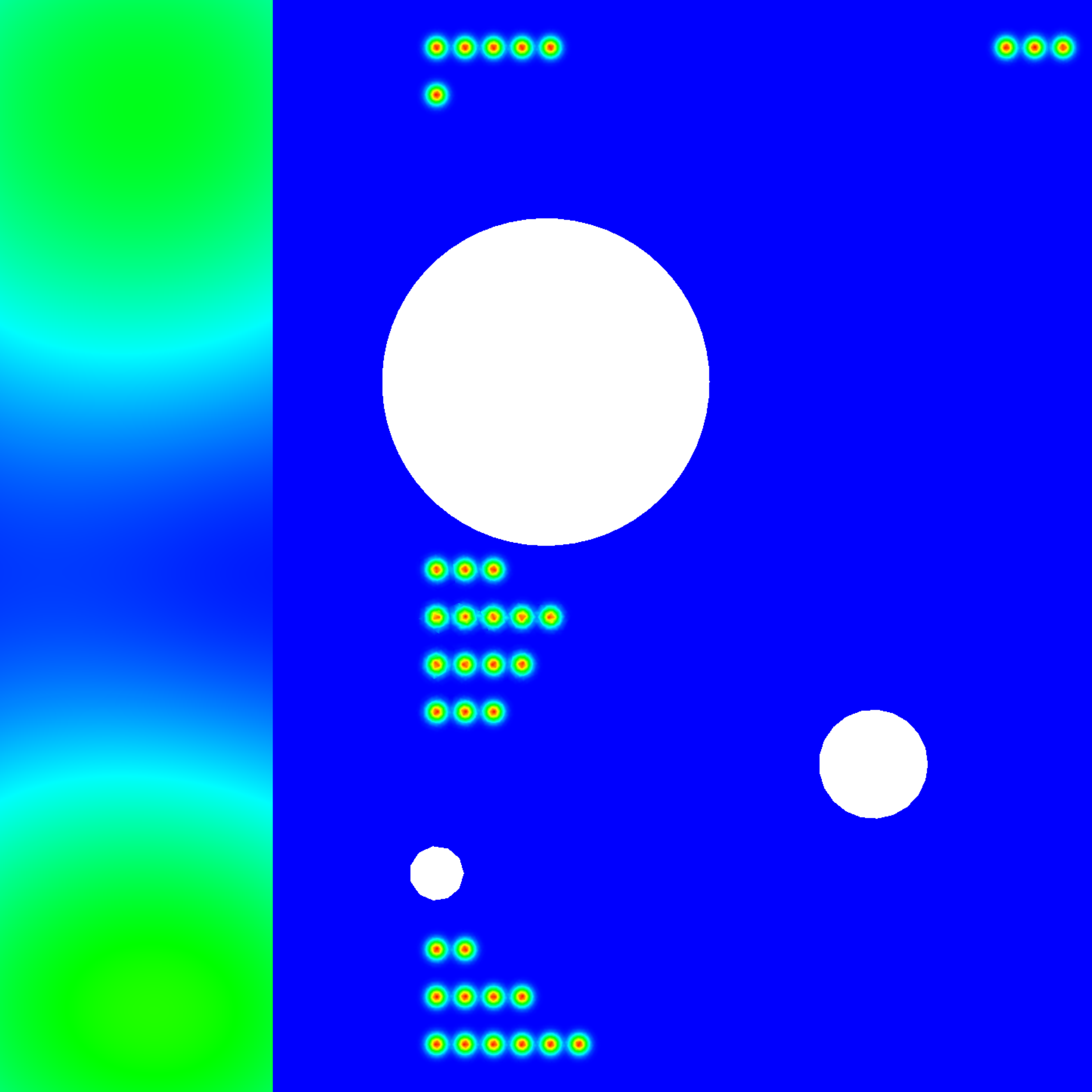}
    \includegraphics[width=0.32\linewidth]{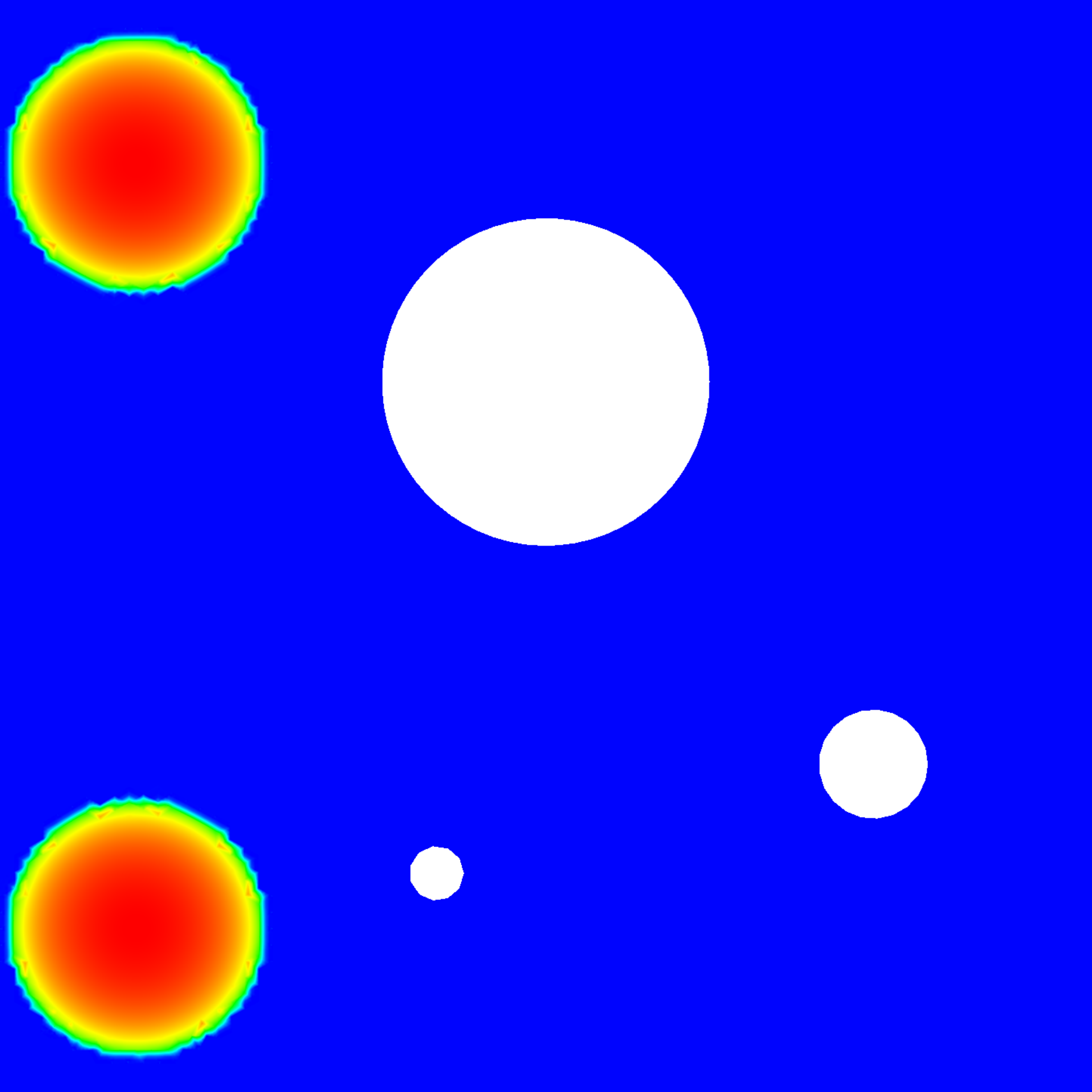}
    \includegraphics[width=0.32\linewidth]{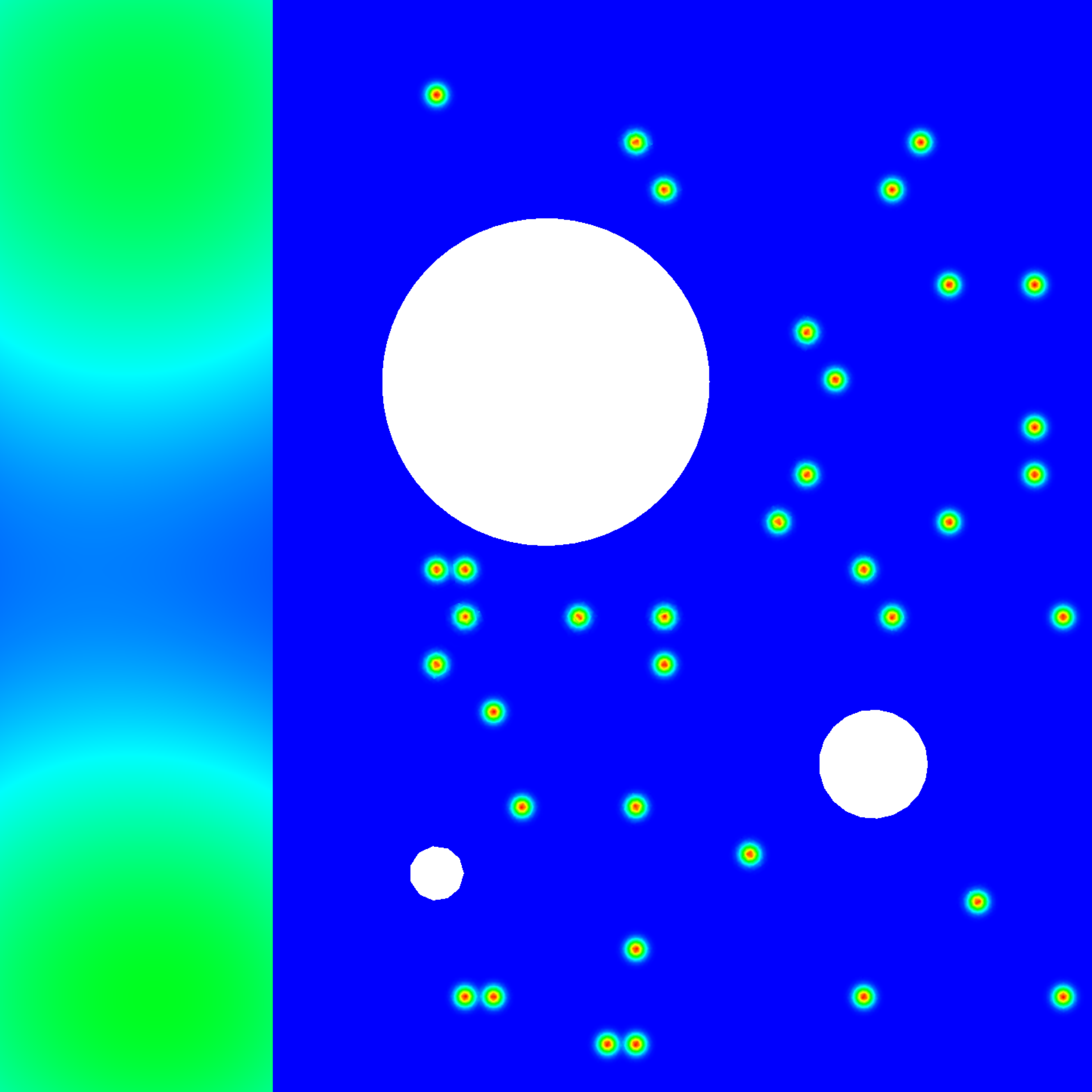}
    \caption{Left: Maximum a posteriori estimate $s_{\overline{\w}^{m_0}}$ with proposed design. Center: True $s$. Right: Maximum a posteriori estimate $s_{\w_{\mathrm{RNG}}^{m_0}}$ with random design.}
    \label{fig:MAPs}
\end{figure}

\section{Conclusions}

In this work, we have demonstrated that the low-rank-based redundant-dominant $p$-continuation algorithm can be applied to obtain binary A-optimal designs for time-dependent PDE-based inverse problems, producing high-quality designs in relatively low time. This suggests the general validity of the method proposed in \cite{Aarset} to a broad class of physically relevant equations, significantly widening the scope of its applicability. However, it is also clear that more informative data is required to obtain good reconstructions with only a small number of sensor placements; as such, the authors next intend to consider also scenarios where multiple measurements are made in each sensor over different time points, in accordance to the setting of \cite[Thm.~17]{Aarset}, particularly in the context of e.g.~trace data. To further accelerate the process of identifying optimal designs, the authors moreover plan on applying the bi-level algorithms  \cite{nguyen-seq, nguyen-seqbi} to obtain further gains in terms of computational speed and robustness.

\begin{acknowledgement}
The authors acknowledge support from the DFG through Grant 432680300 - SFB 1456 (C04).
\end{acknowledgement}
%
%

\begin{thebibliography}{99.}%

\bibitem{Aarset}
Aarset, C.: Global optimality conditions for sensor placement, with extensions to binary A-optimal experimental designs. Inverse Problems, 41(6), Art.~Id.~065013 (2024)

\bibitem{Alexanderian}
Alexanderian, A.: Optimal experimental design for infinite-dimensional Bayesian inverse problems governed by PDEs: a review. Inverse Problems \textbf{37}(4), p. 043001 (2021)	

\bibitem{AttiaConstantinescu}
Attia, A. and Constantinescu, E.: Optimal Experimental Design for Inverse Problems in the Presence of Observation Correlations. SIAM J. on Sci. Compt \textbf{44}(4), pp.~A2808-A2842 (2022)

\bibitem{DaonStadler}
Daon, Y and Stadler, G.: Mitigating the influence of the boundary on PDE-based covariance operators. Inverse Problems Imaging \textbf{12} 1083–1102 (2018)

\bibitem{Evans}
Evans, L: Partial Differential Equations. Graduate Studies in Mathematics 19. AMS, Providence, RI (1998)

\bibitem{HalkoMartinssonTropp}
Halko, N. and Martinsson, P. G. and Tropp, J. A..: Finding Structure with Randomness: Probabilistic Algorithms for Constructing Approximate Matrix Decompositions. SIAM Review \textbf{53}(2), 
pp.~217-288 (2011) 

\bibitem{TCC}
Kaltenbacher, B. and Nguyen, T. T. N. and Scherzer, O.: The tangential cone condition for some coefficient identification model problems in parabolic PDEs. Time-dependent Problems in Imaging and Parameter Identification, 
pp.~121–163 (2021)

\bibitem{Isakov-inversesourceproblem}
Isakov, V.: Inverse Source Problems. American Mathematical Soc., ISBN 0821815326, 9780821815328, 193 pp (1990)

\bibitem{Mercer}
Mercer, J. and Forsyth, A. R.: XVI. Functions of positive and negative type, and their connection the theory of integral equations. Philosophical Transactions of the Royal Society of London. Series A, Containing Papers of a Mathematical or Physical Character \textbf{209}(441-458), {doi:10.1098/rsta.1909.0016}, pp.~415-446 (1909) 

\bibitem{nguyen19}
Nguyen, T. T. N.: Landweber-Kaczmarz for parameter identification in time-dependent inverse problems: All-at-once versus Reduced version. Inverse Problems, 35(3), Art.~Id.~035009 (2019)		

\bibitem{nguyen-seq}
Nguyen, T. T. N: Bi-level iterative regularization for inverse problems in nonlinear PDEs. Inverse Problems \textbf{40}(4), 36pp (2024)	

\bibitem{nguyen-seqbi}
Nguyen, T. T. N: Sequential bi-level regularized inversion
with application to hidden reaction law discovery. Inverse Problems \textbf{41}(6), 34pp (2025)

\bibitem{Roubicek}
{Roub\' i\v cek}, T.: Nonlinear Partial Differential Equations with Applications. Springer Basel, ISBN 978-3-0348-0513-1, 416 pp (2013)

\bibitem{Schoberl}
Sch{\"o}berl, J.: C++ 11 implementation of finite elements in NGSolve.
Institute for analysis and scientific computing, Vienna University of Technology \textbf{30} (2014)

\bibitem{Stuart}
Stuart, A.: Inverse problems: A Bayesian perspective.
Acta Numerica pp.~451--559, Cambridge University Press (2010)


%
%
%
%
%
%
%


\end{thebibliography}
%

\end{document}